\documentclass[11pt,reqno]{amsart}
\usepackage[left=2cm,right=2cm,top=3cm,bottom=3cm]{geometry}
\usepackage{amssymb}
\usepackage{soul}
\usepackage{graphicx}
\usepackage[new]{old-arrows}
\usepackage{euscript}
\usepackage{cite}
\usepackage{leftindex}
\usepackage{color}
\usepackage{tikz-cd}
\usetikzlibrary{tqft}
\usepackage{mathtools}
\usepackage{hyperref}
\hypersetup{
	colorlinks=true,       			% false: boxed links; true: colored links
	linkcolor=blue,          			% color of internal links
	citecolor=blue,		 		% color of links to bibliography
	filecolor=blue,      				% color of file links
	urlcolor=blue           			% color of external links
}
\usepackage{amsmath}
\usepackage{amsthm}
\usepackage{tikz}
\usepackage{amscd}
\usepackage{latexsym}
\usepackage{epsfig}
\usepackage{graphicx}
\usepackage{caption}
\usepackage{rotating}
\usepackage{mathtools}
\usepackage{xypic}
\usepackage{bigints} 
\usepackage{longtable}
\usepackage{extpfeil}
\usepackage{tensor}

\newcommand{\g}{\mathfrak{g}}
\renewcommand{\c}{\mathfrak{c}}

\newcommand{\ad}{\mathrm{ad}}

\newcommand{\sll}[1]{\mkern-4mu\mathbin{/\mkern-5mu/}_{\mkern-4mu{#1}}}

\newcommand{\Lim}[1]{\raisebox{0.5ex}{\scalebox{1.0}{$\displaystyle \lim_{#1}\;$}}}

\newcommand\junk[1]{}

\numberwithin{equation}{section}

\newtheorem{theorem}{Theorem}[section]
\newtheorem{proposition}[theorem]{Proposition}
\newtheorem{corollary}[theorem]{Corollary}
\newtheorem{lemma}[theorem]{Lemma}

\newtheorem{mtheorem}{Main Theorem}

\theoremstyle{definition}
\newtheorem{definition}[theorem]{Definition}

\newtheorem{remark}[theorem]{Remark}

\title[Poisson blow-ups and the adjoint quotient]{Poisson blow-ups and the adjoint quotient}

\author[Peter Crooks]{Peter Crooks}
\author[Iva Halacheva]{Iva Halacheva}
\address[Peter Crooks]{Department of Mathematics and Statistics\\ Utah State University \\ 3900 Old Main Hill \\ Logan, UT 84322, USA}
\email{peter.crooks@usu.edu}
\address[Iva Halacheva]{Department of Mathematics \\ Northeastern University \\ 360 Huntington Avenue \\ Boston, MA 02115, USA}
\email{i.halacheva@northeastern.edu}

\begin{document}
	
	\begin{abstract} 
		We leverage Polishchuk's Poisson blow-up criterion in the context of algebro-geometric integrable systems. In more detail, one may associate an integrable system $\tau:\mathfrak{X}\longrightarrow\mathfrak{B}$ to each affine Poisson scheme $\mathfrak{X}$ over $\mathbb{C}$. We prove that the blow-ups of $\mathfrak{X}$ along fibers of $\tau$ are Poisson schemes occurring in a family $\widetilde{\mathfrak{X}\times\mathfrak{B}}\longrightarrow\mathfrak{B}$, where $\widetilde{\mathfrak{X}\times\mathfrak{B}}$ is itself a Poisson scheme. This result is subsequently specialized to the adjoint quotient $\tau:\g\longrightarrow\g\sll{}{G}\eqqcolon\mathfrak{c}$ of a finite-dimensional complex semisimple Lie algebra $\g$ with integrating algebraic group $G$. We show that the family $\widetilde{\g\times\mathfrak{c}}\longrightarrow\mathfrak{c}$ is flat, conical, and equipped with a canonical Poisson Hamiltonian $G$-variety structure. We also obtain Poisson-geometric results on the fibers of this family, which are blow-ups of $\g$ along regular adjoint orbit closures. 
	\end{abstract}
	
	\maketitle
	
	\tableofcontents

	\section{Introduction}
	
	\subsection{Motivation and context} The blow-up construction is ubiquitous in both classical and modern algebraic geometry. Its connection to algebraic Poisson geometry was initiated by Polishchuk \cite{pol:97} in 1997, and features in several works over the last decade \cite{pym:18,bai-cav-van:19,van:18,che-che-yan-yan:25,geu-zam:21,lap-mat-pym-zup:26,lin-pym:24}. At the heart of this connection is \textit{Polishchuk's criterion} for the blow-up of a Poisson scheme $X$ along a quasi-coherent ideal sheaf $\mathcal{I}\subseteq\mathcal{O}_X$ to be Poisson: if $\{\mathcal{O}_X,\mathcal{I}\}\subseteq\mathcal{I}$ and $\{\mathcal{I},\mathcal{I}\}\subseteq\mathcal{I}^2$, then the blow-up carries a unique Poisson structure for which the blow-up morphism is Poisson and the exceptional divisor is a Poisson subscheme. Polishchuk also shows the converse to be true under mild hypotheses.
	
	In light of the above, it is natural to seek interesting Poisson schemes arising from Polishchuk's criterion. Some examples are immediate: the blow-up of a smooth two-dimensional Poisson variety at a point is Poisson if and only if the Poisson structure vanishes at the point. Recently studied examples include del Pezzo surfaces, as well as blow-ups of certain Calabi--Yau six-folds \cite{che-che-yan-yan:25}. On the other hand, there are few concrete techniques for constructing examples. There also appear to be few Lie-theoretic examples, e.g. Poisson varieties occurring as blow-ups of the dual of a finite-dimensional Lie algebra. The central purpose of this work is to address these two issues.
	
	\subsection{Main results}
	We associate an integrable system $\tau:\mathfrak{X}\longrightarrow\mathfrak{B}$ to each affine Poisson scheme $\mathfrak{X}$ over $\mathbb{C}$. This leads us to consider the blow-ups $\mathrm{Bl}_{\tau^{-1}(b)}(\mathfrak{X})\longrightarrow\mathfrak{X}$ of $\mathfrak{X}$ along the fibers $\tau^{-1}(b)$, where $b\in\mathfrak{B}$ ranges over closed points. To investigate these blow-ups systematically, we consider the graph $\Gamma\coloneqq\mathfrak{X}\times_{\mathfrak{B}}\mathfrak{B}\subseteq\mathfrak{X}\times\mathfrak{B}$ and blow-up $\widetilde{X\times\mathfrak{B}}\coloneqq\mathrm{Bl}_{\Gamma}(\mathfrak{X}\times\mathfrak{B})\longrightarrow\mathfrak{X}\times\mathfrak{B}$. Let $\pi_{\mathfrak{B}}:\widetilde{X\times\mathfrak{B}}\longrightarrow\mathfrak{B}$ denote the result of composing $\widetilde{\mathfrak{X}\times\mathfrak{B}}\longrightarrow\mathfrak{X}\times\mathfrak{B}$ with the projection $\mathfrak{X}\times\mathfrak{B}\longrightarrow\mathfrak{B}$.
	
	\begin{mtheorem}\label{Theorem: Main Theorem 1}
		The following statements are true.
		\begin{itemize}
			\item[\textup{(i)}] For a closed point $b\in\mathfrak{B}$, the pair $(\mathfrak{X},\tau^{-1}(b))$ satisfies Polishchuk's criterion. In particular, $\mathrm{Bl}_{\tau^{-1}(b)}(\mathfrak{X})$ is Poisson and its exceptional divisor is a Poisson subscheme.
			\item[\textup{(ii)}] The pair $(\mathfrak{X}\times\mathfrak{B},\Gamma)$ satisfies Polishchuk's criterion. In particular, $\widetilde{\mathfrak{X}\times\mathfrak{B}}$ is Poisson and its exceptional divisor is a Poisson subscheme.
			\item[\textup{(iii)}] For a closed point $b\in\mathfrak{B}$, there is a canonical scheme isomorphism $\pi_{\mathfrak{B}}^{-1}(b)\cong\mathrm{Bl}_{\tau^{-1}(b)}(\mathfrak{X})$. Furthermore, the following composite morphism is Poisson: $$\mathrm{Bl}_{\tau^{-1}(b)}(\mathfrak{X})\overset{\cong}\longrightarrow\pi_{\mathfrak{B}}^{-1}(b)\longhookrightarrow\widetilde{\mathfrak{X}\times\mathfrak{B}}.$$
		\end{itemize}
	\end{mtheorem}
	
	This result may be summarized informally as follows: the Poisson blow-ups $\mathrm{Bl}_{\tau^{-1}(b)}$ form a family $$\pi_{\mathfrak{B}}:\widetilde{X\times\mathfrak{B}}\longrightarrow\mathfrak{B},$$ where $\widetilde{\mathfrak{X}\times\mathfrak{B}}$ is Poisson and contains the blow-ups $\mathrm{Bl}_{\tau^{-1}(b)}$ as Poisson subschemes.
	
	We now specialize to the case $\mathfrak{X}=\mathfrak{g}^*=\mathfrak{g}$, where $\g$ is a finite-dimensional complex semisimple Lie algebra, and the Killing form is used to identify $\g^*$ with $\g$. The integrable system $\tau:\mathfrak{X}\longrightarrow\mathfrak{B}$ becomes the adjoint quotient $\tau:\g\longrightarrow\g\sll{}{G}\eqqcolon\mathfrak{c}$. Its fibers are precisely the closures of regular adjoint orbits in $\g$. One such fiber is the nilpotent cone $\mathcal{N}\subseteq\g$; it features prominently in what follows.
	
	Using Main Theorem \ref{Theorem: Main Theorem 1}, we obtain the Poisson blow-ups $\widetilde{\g\times\mathfrak{c}}\longrightarrow\mathfrak{g}\times\mathfrak{c}$ and $\mathrm{Bl}_{\tau^{-1}(c)}(\g)\longrightarrow\g$ for $c\in\mathfrak{c}$. We also have the morphisms $\pi_{\g}:\widetilde{\g\times\c}\longrightarrow\g$ and $\pi_{\mathfrak{c}}:\widetilde{\g\times\mathfrak{c}}\longrightarrow\mathfrak{c}$, obtained by composing $\widetilde{\g\times\mathfrak{c}}\longrightarrow\mathfrak{g}\times\mathfrak{c}$ with the projections $\g\times\mathfrak{c}\longrightarrow\g$ and $\g\times\mathfrak{c}\longrightarrow\c$, respectively.
	
	\begin{mtheorem}\label{Theorem: Main Theorem 2}
		The following statements are true.
		\begin{itemize}
			\item[\textup{(i)}] The morphism $\pi_{\mathfrak{c}}:\widetilde{\g\times\mathfrak{c}}\longrightarrow\mathfrak{c}$ is flat.
			\item[\textup{(ii)}] There are algebraic $\mathbb{C}^{\times}$-actions on $\widetilde{\g\times\mathfrak{c}}$ and $\mathfrak{c}$ that make $\pi_{\mathfrak{c}}$ equivariant, contract $\mathfrak{c}$ to the point $\tau(\mathcal{N})\in\mathfrak{c}$, and contract $\widetilde{\g\times\mathfrak{c}}$ to $\mathrm{Bl}_{\mathcal{N}}(\g)$.
			\item[\textup{(iii)}] The $G$-action on $\g\times\mathfrak{c}$ has a unique lift to a Poisson Hamiltonian $G$-variety structure on $\widetilde{\g\times\mathfrak{c}}$, where $G$ acts on $\g\times\mathfrak{c}$ by the adjoint action on the first factor. The corresponding moment map is $\pi_{\g}:\widetilde{\g\times\mathfrak{c}}\longrightarrow\g$.
		\end{itemize}
	\end{mtheorem}
	
	We also obtain results on the blow-ups of $\g$ along specific regular adjoint orbit closures.
	
	\subsection{Organization} Each section begins with a summary of its contents. Main Theorem \ref{Theorem: Main Theorem 1} is proved as Theorem~\ref{Theorem:IntegrableSystem} in the main text. Main Theorem \ref{Theorem: Main Theorem 2} is proved as Theorem \ref{Theorem: Main theorem Lie algebra} and Propositions \ref{Proposition: Actions} and \ref{Proposition: Choice}.
	
	\subsection*{Acknowledgements} We thank Tom Gannon, Sam Gunningham, and Valerio Toledano Laredo for illuminating discussions. P.C. was supported by the National Science Foundation Grant No. DMS-2454103 and Simons Foundation Grant No. MPS-TSM-00002292. I.H. was supported by the National Science Foundation Grant No. DMS-2302664 and Simons Foundation Grant No. MPS-TSM-00026099. This material is based upon work supported by the National Science Foundation under Grant No. DMS-2424139, while I.H. was in residence at the Simons Laufer Mathematical Sciences Institute in Berkeley, California, during the Fall 2026 semester.

	\section{Poisson-geometric results}\label{Section: Poisson-geometric results}
	This section develops the main necessary results in algebraic Poisson geometry. In Section \ref{Subsection: Preliminaries}, we collect some straightforward facts about the regular locus and Poisson rank of a potentially singular Poisson variety. Section \ref{Subsection: Polishchuk} then recalls some results of Polishchuk on Poisson blow-ups, and establishes facts about their regular loci and Poisson ranks. In Section \ref{Subsection: A canonical integrable system}, we associate a canonical integrable system $\tau:\mathfrak{X}\longrightarrow\mathfrak{B}$ to an affine Poisson scheme $\mathfrak{X}$. This leads to Section \ref{Subsection: A family}, where we consider the Poisson blow-ups of $\mathfrak{X}$ along the fibers of $\tau$. We show that these blow-ups form a family $\widetilde{\mathfrak{X}\times\mathfrak{B}}\longrightarrow\mathfrak{B}$ possessing several Poisson-geometric properties.
	
	\subsection{Preliminaries on Poisson varieties and schemes}\label{Subsection: Preliminaries} We begin by setting conventions on symplectic leaves in algebraic Poisson varieties. A first step is to recall the essentials of symplectic leaves in holomorphic Poisson geometry. To this end, let $X$ be a complex manifold with structure sheaf $\mathcal{O}_X$. One calls $X$ a \textit{holomorphic Poisson manifold} if $\mathcal{O}_X$ has been enriched to a sheaf of Poisson algebras. In this case, there exists a unique holomorphic bivector field $\sigma\in\mathrm{H}^0(X,\bigwedge^2(TX))$ satisfying $\{f,g\}=\sigma(\mathrm{d}f,\mathrm{d}g)$ for all $f,g\in\mathcal{O}_X$. The image of the vector bundle morphism 
	$$\sigma^{\vee}:T^*X\longrightarrow TX,\quad (x,\phi)\mapsto\sigma_x(\phi,\cdot)\in T_xX$$
	is a (potentially singular) integrable holomorphic distribution on $X$. At the same time, $\sigma$ is tangent to each integral leaf $L\subseteq X$ of this distribution. The resulting holomorphic Poisson structure on $L$ turns out to be symplectic: there exists a unique holomorphic symplectic form $\omega_L\in\mathrm{H}^0(L,\bigwedge^2(T^*L))$ satisfying $(\omega_L^{\vee})^{-1}=-\sigma_L^{\vee}$, where $\sigma_L$ is the Poisson bivector field on $L$ and $\omega_L^{\vee}$ is defined by
	$$\omega_L^{\vee}:TL\longrightarrow T^*L,\quad (x,v)\mapsto(\omega_L)_x(v,\cdot)\in T_x^*L.$$ The holomorphic symplectic manifolds $(L,\omega_L)$ are called the \textit{symplectic leaves} of $X$.
	
	A \textit{Poisson scheme} is a scheme over $\mathbb{C}$ whose structure sheaf has been enriched to one of Poisson algebras. Assume that $X$ is a Poisson variety, by which we mean that $X$ is a reduced Poisson scheme of finite type. Set $X_0\coloneqq X$, and let $X_{i+1}$ denote the singular locus $(X_i)_{\text{sing}}$ of $X_i$ for all $i\in\mathbb{Z}_{\geq 0}$. The resulting chain $X_0\supseteq X_1\supseteq\cdots$ of closed subvarieties yields $\emptyset$ after finitely many steps. Let $n\in\mathbb{Z}_{\geq 0}$ be such that $X_n=\emptyset$ and $X_i\neq\emptyset$ for $i<n$. It follows that $$X=X_0\supseteq X_1\supseteq\cdots\supseteq X_n=\emptyset$$ is a strictly decreasing filtration of $X$ by closed subvarieties. One also knows that $X_i$ is a Poisson subvariety of $X$ for all $i\in\{0,\ldots,n\}$ \cite[Corollary 2.4]{pol:97}. Setting $X_i^{\circ}\coloneqq X_i\setminus (X_i)_{\text{sing}}$ for all $i\in\{1,\ldots,n\}$, we observe that
	\begin{equation}\label{Equation: Smooth decomposition}X=\bigcup_{i=0}^nX_i^{\circ}\end{equation} expresses $X$ as a disjoint union of smooth, locally closed, Poisson subvarieties. Each subvariety $X_i^{\circ}\subseteq X$ may be regarded as a holomorphic Poisson manifold. As such, it decomposes into symplectic leaves. A subset $L\subseteq X$ is called a \textit{symplectic leaf} if $L$ is a symplectic leaf of $X_i^{\circ}$ for some $i\in\{0,\ldots,n\}$. This implies that each $x\in X$ is contained in a unique symplectic leaf, to be denoted $L_x\subseteq X$. The \textit{Poisson rank} $\mathrm{Prk}(X)$ of $X$ is the maximum symplectic leaf dimension occurring in $X$. One may consider the \textit{regular locus}
	$$X_{\text{reg}}\coloneqq\{x\in X:\dim(L_x)=\mathrm{Prk}(X)\}.$$  
	
	The following is well-known if $X$ is smooth, and perhaps known to experts in its stated generality.
	
	\begin{proposition}\label{Proposition: Open}
		If $X$ is a Poisson variety, then $X_{\emph{reg}}$ is open in $X$.
	\end{proposition}
	
	\begin{proof}
		Choose an open cover $$X=\bigcup_{j=1}^k Y_j$$ of $X$ by affine Poisson subvarieties $Y_j=\mathrm{Spec}(A_j)$. It suffices to prove that $X_{\text{reg}}\cap Y_j$ is open in $Y_j$ for all $j\in\{1,\ldots,k\}$. To this end, suppose that $j\in\{1,\ldots,k\}$. Choose generators $z_1,\ldots,z_{\ell}$ of $A_j$. For $y\in Y_j$, the dimension of $L_y$ is the rank of the following matrix $M(y)\in\mathrm{Mat}_{\ell\times \ell}(\mathbb{C})$ \cite[Section 3 and Proposition 3.6]{bro-gor:03}: $$M(y)_{pq}\coloneqq[\{z_p,z_q\}]\in\mathcal{O}_{X,y}/{\mathfrak{m}_y}=\mathbb{C}$$ for all $p,q\in\{1,\ldots,\ell\}$, where $\mathcal{O}_{X,y}$ is the stalk of $\mathcal{O}_X$ at $y$ and $\mathfrak{m}_y\subseteq\mathcal{O}_{X,y}$ is the unique maximal ideal. This implies that $\mathrm{rank}(M(y))\leq\mathrm{Prk}(X)$ for all $y\in Y_j$, so that $$\{y\in Y_j:\mathrm{rank}(M(y))\geq\mathrm{Prk}(X)\}=\{y\in Y_j:\mathrm{rank}(M(y))=\mathrm{Prk}(X)\}.$$ The left-hand side is evidently open in $Y_j$, while the right-hand side is $X_{\text{reg}}\cap Y_j$.
	\end{proof}
	
	We continue with the notation introduced above, in which $X$ denotes a Poisson variety with structure sheaf $\mathcal{O}_X$. Given $f \in \mathcal{O}_X$, one defines the Hamiltonian vector field $\mathrm{H}(f)\coloneqq\{f,\cdot\} \in \mathrm{Der}(\mathcal{O}_X)$ on $X$. One has $$T_x(L_x)=\{\mathrm{H}(f)_x:f\in\mathcal{O}_{X,x}\}\subseteq T_xX$$ for each $x\in X$, where $\mathcal{O}_{X,x}$ denotes the stalk of $\mathcal{O}_X$ at $x$. As with Proposition \ref{Proposition: Open}, the following is well-known when $X$ is smooth; we include its proof for completeness.
	
	\begin{lemma}\label{Lemma: Inequality}
		Suppose that $X$ and $Y$ are Poisson varieties. Consider a Poisson variety morphism $\mu:X\longrightarrow Y$, a point $x\in X$, and the symplectic leaves $L_x\subseteq X$ and $L_{\mu(x)}\subseteq Y$. We then have $\dim(L_x)\geq\dim(L_{\mu(x)})$. If equality holds, then $T_x(L_x)=\{\mathrm{H}(\mu^*(f))_x:f\in\mathcal{O}_{Y,\mu(x)}\}$.
	\end{lemma}
	
	\begin{proof}
		Choose germs $f_1,\ldots,f_n\in\mathcal{O}_{Y,\mu(x)}$ such that $\{\mathrm{H}(f_i)_{\mu(x)}\}_{i=1}^n$ is a basis of $T_{\mu(x)}(L_{\mu(x)})$. Let us also consider the germs $\mu^*(f_1),\ldots,\mu^*(f_n)\in\mathcal{O}_{X,x}$. We have $$\mathrm{d}\mu_x(\mathrm{H}(\mu^*(f_i))_x)=\mathrm{H}(f_i)_{\mu(x)}\quad\text{and}\quad \mathrm{H}(\mu^*(f_i))_x\in T_x(L_{x})$$ for all $i\in\{1,\ldots,n\}$. It follows that $\{\mathrm{H}(\mu^*(f_i))_x\}_{i=1}^n$ is a linearly independent subset of $T_x(L_x)$. This implies that $\dim(T_x(L_x))\geq n=\dim(T_{\mu(x)}(L_{\mu(x)}))$, i.e. $\dim(L_x)\geq\dim(L_{\mu(x)})$. 
		
		If $\dim(T_x(L_x))=\dim(T_{\mu(x)}(L_{\mu(x)}))$, then $\{\mathrm{H}(\mu^*(f_i))_x\}_{i=1}^n$ is a basis of $T_x(L_x)$. It follows that $T_x(L_x)=\{\mathrm{H}(\mu^*(f))_x:f\in\mathcal{O}_{Y,\mu(x)}\}$ in this case. 
	\end{proof}
	
	\subsection{Poisson blow-ups and Polishchuk's criterion}\label{Subsection: Polishchuk}
	We now recall the blow-up of a complex scheme $X$ along a closed subscheme $Y \subseteq X$. Write $\mathcal{I}_Y\subseteq\mathcal{O}_X$ for the ideal sheaf of $Y$, and $\underline{\mathbb{C}}$ (resp. $\underline{\mathbb{C}[t]}$) for the constant sheaf on $X$ with value $\mathbb{C}$ (resp. $\mathbb{C}[t]$). Note that the \textit{Rees algebra} $$\mathrm{Rees}_{\mathcal{I}_Y}(\mathcal{O}_X)\coloneqq\bigoplus_{j=0}^{\infty}(\mathcal{I}_Y)^jt^j\subseteq\mathcal{O}_X\otimes_{\underline{\mathbb{C}}}\underline{\mathbb{C}[t]}$$ is a sheaf of $\mathbb{Z}_{\geq 0}$--graded commutative $\mathcal{O}_X$-algebras. The \textit{blow-up} of $X$ along $Y$ is the global Proj of this sheaf, namely
	$$\mathrm{Bl}_Y(X)\coloneqq\mathrm{Proj}(\mathrm{Rees}_{\mathcal{I}_Y}(\mathcal{O}_X))\overset{\pi}\longrightarrow X.$$ The inverse image of $Y$ is a divisor called the \textit{exceptional divisor},  $\mathcal{E}_Y(X)\coloneqq\pi^{-1}(Y)\subseteq\mathrm{Bl}_Y(X)$, and the restriction of $\pi$ to $\mathrm{Bl}_Y(X)\setminus\mathcal{E}_Y(X)$ is an isomorphism of schemes, \begin{equation}\label{Equation: Isomorphism off divisor}\mathrm{Bl}_Y(X)\setminus\mathcal{E}_Y(X)\overset{\cong}\longrightarrow X\setminus Y.\end{equation}
	Note that if $X$ and $Y$ are smooth, then so are $\mathrm{Bl}_Y(X)$ and $\mathcal{E}_Y(X)$.
	
	Polishchuk studies the blow-up construction in Poisson geometry \cite{pol:97}. In more detail, consider a Poisson scheme $X$ and closed subscheme $Y\subseteq X$. A natural question is whether $\mathrm{Bl}_Y(X)$ carries a Poisson structure for which the blow-up morphism $\pi:\mathrm{Bl}_Y(X)\longrightarrow X$ is Poisson; it is clearly unique if it exists. To refine this question, equip $\mathrm{Bl}_Y(X)\setminus\mathcal{E}_Y(X)$ with the unique Poisson structure for which \eqref{Equation: Isomorphism off divisor} is Poisson. The question is now whether the Poisson structure on $\mathrm{Bl}_Y(X)\setminus\mathcal{E}_Y(X)$ extends across the exceptional divisor. This need not occur, as can be witnessed by letting $X$ be a smooth two-dimensional Poisson variety and setting $Y=\{x\}$ for a point $x\in X$. The Poisson structure on $\mathrm{Bl}_{x}(X)$ extends across the exceptional divisor if and only if $\sigma(x)=0$, where $\sigma$ is the Poisson bivector field on $X$ \cite[Proposition 2.19]{pym:18}. This example suggests requiring $Y$ to be a closed Poisson subscheme of $X$, i.e. that the ideal sheaf $\mathcal{I}_Y\subseteq\mathcal{O}_X$ of $Y$ should be a sheaf of Poisson ideals.
	
	Suppose that $\mathcal{I}\subseteq\mathcal{O}_X$ is a Poisson ideal sheaf. Polishchuk calls $\mathcal{I}\subseteq\mathcal{O}_X$ \textit{degenerate} if $$\{f,g\}h+\{g,h\}f+\{h,f\}g\in\mathcal{I}^3$$ for all $f,g,h\in\mathcal{I}$. Let us call $\mathcal{I}$ \textit{strongly coisotropic} if $\{\mathcal{I},\mathcal{I}\}\subseteq\mathcal{I}^2$. This condition is clearly stronger than $\mathcal{I}$ being degenerate. With this in mind, we summarize \cite[Theorem 8.2 and Proposition 8.3]{pol:97} as follows.    
	
	\begin{theorem}[Polishchuk]\label{Theorem: Polishchuk}
		Consider a Poisson scheme $X$ and closed Poisson subscheme $Y\subseteq X$. 
		\begin{itemize}
			\item[\textup{(i)}] If $\mathcal{I}_Y$ is degenerate, then there exists a unique Poisson structure on $\mathrm{Bl}_Y(X)$ for which the blow-up morphism $\pi:\mathrm{Bl}_Y(X)\longrightarrow X$ is Poisson.
			\item[\textup{(ii)}] If $\mathcal{I}_Y$ is strongly coisotropic, then $\mathcal{E}_{Y}(X)$ is a Poisson subscheme of $\mathrm{Bl}_Y(X)$.
			\item[\textup{(iii)}] If $X$ and $Y$ are smooth and $\mathrm{Bl}_Y(X)$ carries a Poisson structure for which $\pi:\mathrm{Bl}_Y(X)\longrightarrow X$ is Poisson, then $\mathcal{I}_Y$ is degenerate. If one also assumes that $\mathcal{E}_Y(X)$ is a Poisson subscheme of $\mathrm{Bl}_Y(X)$, then $\mathcal{I}_Y$ is strongly coisotropic.
		\end{itemize}
	\end{theorem}
	
	\begin{definition}\label{Definition: Polishchuk}
		We say that a pair $(X,Y)$ of a Poisson scheme $X$ and closed subscheme $Y \subseteq X$ satifies \textit{Polishchuk's criterion} if the ideal sheaf $\mathcal{I}_Y$ of $Y$ is Poisson and strongly coisotropic. In particular, the conclusions in Theorem~\ref{Theorem: Polishchuk}(i),(ii) follow.
		
	\end{definition}
	\begin{remark}\label{Remark: Affine open}
		We now discuss the proof of Theorem~\ref{Theorem: Polishchuk}(i). A first step is to regard elements of $\mathcal{I}_Y\setminus\{0\}$ as degree-one elements in $R\coloneqq\mathrm{Rees}_{\mathcal{I}_Y}(\mathcal{O}_X)$. Each $f\in\mathcal{I}_Y\setminus\{0\}$ thereby determines the affine open subset $D_{+}(f)\subseteq\mathrm{Bl}_Y(X)$. One finds that
		$$\mathrm{Bl}_Y(X)=\bigcup_{f\in\mathcal{I}_Y\setminus\{0\}}D_{+}(f).$$ On the other hand, consider the localization $R_f$ of $R$ by the powers of $f\in\mathcal{I}_Y\setminus\{0\}$; it is a graded $\mathbb{C}$-algebra. Write $R_{(f)}\subseteq R_f$ for the graded component in degree zero, noting that $D_{+}(f)=\mathrm{Spec}(R_{(f)})$. Polishchuk proves that $R_{(f)}$ is a Poisson algebra, i.e. $D_{+}(f)$ is a Poisson scheme. It is also clear from the proof of \cite[Theorem 8.2]{pol:97} that the Poisson scheme structures on $D_{+}(f)$ and $D_{+}(g)$ coincide on $D_{+}(f)\cap D_{+}(g)$ for all $f,g\in\mathcal{I}_Y\setminus\{0\}$.
	\end{remark}
	
	We now relate Theorem \ref{Theorem: Polishchuk} to the notions of \textit{Poisson rank} and \textit{regular locus} discussed in Section \ref{Subsection: Preliminaries}. As we have formulated these notions for Poisson varieties instead of Poisson schemes, we work in the former context.
	
	\begin{proposition}\label{Proposition: Poisson expansion}
		Consider an irreducible Poisson variety $X$ and closed Poisson subvariety $Y\subseteq X$ with $\dim (Y)<\dim (X)$. Assume that $\mathrm{Bl}_Y(X)$ carries a Poisson structure for which the blow-up morphism $\pi:\mathrm{Bl}_Y(X)\longrightarrow X$ is Poisson, and that $\mathcal{E}_Y(X)\subseteq \mathrm{Bl}_Y(X)$ is a Poisson subvariety with respect to this structure. The following statements hold:
		\begin{itemize}
			\item[\textup{(i)}] $\mathrm{Prk}(\mathrm{Bl}_Y(X))=\mathrm{Prk}(X)$;
			\item[\textup{(ii)}] $\pi^{-1}(X_{\emph{reg}})\subseteq \mathrm{Bl}_Y(X)_{\emph{reg}}$;
			\item[\textup{(iii)}] if $x\in\pi^{-1}(X_{\emph{reg}})$, then $T_x(L_x)=\{\mathrm{H}(\pi^*(f))_x:f\in\mathcal{O}_{X,\pi(x)}\}$;
			\item[\textup{(iv)}] if $Y\subseteq X_{\emph{reg}}$, then $\mathrm{Bl}_Y(X)_{\emph{reg}}=\pi^{-1}(X_{\emph{reg}})=\pi^{-1}(X_{\emph{reg}}\setminus Y)\cup\mathcal{E}_Y(X)$.
		\end{itemize}
	\end{proposition}
	
	\begin{proof}
		We begin by proving (i). By Proposition \ref{Proposition: Open}, $\mathrm{Bl}_Y(X)_{\text{reg}}$ is open in $\mathrm{Bl}_Y(X)$. It therefore suffices to exhibit an open dense Poisson subvariety $U\subseteq \mathrm{Bl}_Y(X)$ satisfying $\mathrm{Prk}(U)=\mathrm{Prk}(X)$. To this end, recall that $\pi$ restricts to a Poisson variety isomorphism $\mathrm{Bl}_Y(X)\setminus\mathcal{E}_Y(X)\longrightarrow X\setminus Y$. This implies that $\mathrm{Prk}(\mathrm{Bl}_Y(X)\setminus\mathcal{E}_Y(X))=\mathrm{Prk}(X\setminus Y)$, and it is clear from Proposition \ref{Proposition: Open} that $\mathrm{Prk}(X\setminus Y)=\mathrm{Prk}(X)$. We may therefore take $U=\mathrm{Bl}_Y(X)\setminus\mathcal{E}_Y(X)$.
		
		To prove (ii), suppose that $x\in\pi^{-1}(X_{\text{reg}})$. We have $$\mathrm{Prk}(\mathrm{Bl}_Y(X))\geq\dim(L_{x})\geq\dim(L_{\pi(x)})=\mathrm{Prk}(X).$$ where the second inequality follows from Lemma \ref{Lemma: Inequality}. Part (i) now implies that $\dim(L_x)=\mathrm{Prk}(\mathrm{Bl}_Y(X))$, i.e. $x\in \mathrm{Bl}_Y(X)_{\text{reg}}$. This yields (ii). By combining Lemma \ref{Lemma: Inequality} with the fact that $\dim(L_x)=\dim(L_{\pi(x)})$, we obtain (iii).
		
		To prove (iv), we assume that $Y\subseteq X_{\text{reg}}$. Part (ii) implies that $\mathcal{E}_Y(X)=\pi^{-1}(Y)\subseteq \mathrm{Bl}_Y(X)_{\text{reg}}$. Recalling again that $\pi$ restricts to a Poisson variety isomorphism $\mathrm{Bl}_Y(X)\setminus\mathcal{E}_Y(X)\longrightarrow X\setminus Y$, we must also have $\mathrm{Bl}_Y(X)_{\text{reg}}\setminus\mathcal{E}_Y(X)=\pi^{-1}(X_{\text{reg}}\setminus Y)$. The last equality follows since $\pi$ further restricts to an isomorphism on the regular locus. These considerations yield
		\begin{align*}\mathrm{Bl}_Y(X)_{\text{reg}} & =(\mathrm{Bl}_Y(X)_{\text{reg}}\setminus\mathcal{E}_Y(X))\cup(\mathrm{Bl}_Y(X)_{\text{reg}}\cap\mathcal{E}_Y(X))\\
			& = \pi^{-1}(X_{\text{reg}}\setminus Y)\cup\mathcal{E}_Y(X)\\
			& = \pi^{-1}(X_{\text{reg}}\setminus Y)\cup\pi^{-1}(Y)\\
			& = \pi^{-1}(X_{\text{reg}}).
		\end{align*}
		This completes the proof.
	\end{proof}
	
	\subsection{A canonical integrable system}\label{Subsection: A canonical integrable system} 
	Suppose that $\mathcal{A}$ is a Poisson algebra over $\mathbb{C}$ whose underlying associative product is commutative and unital, and that $\mathcal{Z}(\mathcal{A})\subseteq\mathcal{A}$ is the Poisson center of $\mathcal{A}$. We set $$\mathfrak{X}\coloneqq\mathrm{Spec}(\mathcal{A}) \quad \quad \text{and} \quad \quad \mathfrak{B}\coloneqq\mathrm{Spec}(\mathcal{Z}(\mathcal{A})).$$ Then $\mathfrak{X}$ and $\mathfrak{B}$ are affine Poisson schemes over $\mathbb{C}$, where $\mathfrak{B}$ carries the zero Poisson structure. The inclusion $\mathcal{Z}(\mathcal{A})\subseteq\mathcal{A}$ induces a Poisson variety morphism $$\tau:\mathfrak{X}\longrightarrow \mathfrak{B}$$  that one may view as an integrable system on $\mathfrak{X}$. More generally, for our purposes an \textit{integrable system} is a morphism of affine Poisson schemes over $\mathbb{C}$, $\tau: X \longrightarrow Y$, such that for the algebras of global functions $\mathcal{A}:=\mathcal{O}_X(X)$ and $\mathcal{B}:=\mathcal{O}_Y(Y)$, we have that $\tau^\ast(\mathcal{B})$ is a Poisson-commutative subalgebra of $\mathcal{A}$.

	We ultimately wish to consider blow-ups of $\mathfrak{X}$ along fibers of $\tau:\mathfrak{X}\longrightarrow\mathfrak{B}$.	We start by recording some properties relating to ideals of $\mathcal{A}$, which we will apply to our setting; their proofs are straightforward.
	\begin{proposition}\label{Proposition: Computation}
		Suppose $\mathcal{A}$ is a Poisson algebra as above, and that $\mathcal{I} \subseteq \mathcal{A}$ is an ideal.
		\begin{itemize}
			\item[\textup{(i)}] If $\mathcal{I}$ is generated by a subset of $\mathcal{Z}(\mathcal{A})$, then $\mathcal{I}$ is Poisson and strongly coisotropic.
			\item [\textup{(ii)}] If $\mathcal{I}$ is Poisson and strongly coisotropic, then there exists a unique Poisson algebra structure on $\mathrm{Rees}_{\mathcal{I}}(\mathcal{A})$ satisfying
			$\{at^i,bt^j\}=\{a,b\}t^{i+j}$ for all $i,j\in\mathbb{Z}_{\geq 0}$, $a\in \mathcal{I}^i$, and $b\in \mathcal{I}^j$.
		\end{itemize} 
	\end{proposition}
	Part (ii) of this proposition is used in Section~\ref{Subsection: Quantization}. Part (i) has the following consequence.
	
	\begin{corollary}\label{Corollary: Poisson structure on blow-up}
		Suppose that $Z\subseteq\mathfrak{B}$ is a closed subscheme. Then $(\mathfrak{X},\tau^{-1}(Z))$ satisfies Polishchuk's criterion.
	\end{corollary}
	
	\begin{proof}
		Let $\mathcal{I}\subseteq\mathcal{Z}(\mathcal{A})$ denote the ideal of $Z$ in $\mathfrak{B}$. The ideal of $\tau^{-1}(Z)$ in $\mathfrak{X}$ is $\langle\mathcal{I}\rangle\subseteq\mathcal{A}$. Since $\langle\mathcal{I}\rangle$ is generated by elements of $\mathcal{Z}(\mathcal{A})$, Proposition~\ref{Proposition: Computation}(i) implies that $\langle\mathcal{I}\rangle$ is Poisson and strongly coisotropic.
	\end{proof}
	
	\subsection{A family of Poisson blow-ups}\label{Subsection: A family}
	We continue with the setup of Section \ref{Subsection: A canonical integrable system}. Consider the graph of the integrable system $\tau:\mathfrak{X}\longrightarrow\mathfrak{B}$, i.e. the closed subscheme
	$$\Gamma\coloneqq \mathfrak{X}\times_{\mathfrak{B}}\mathfrak{B}\subseteq\mathfrak{X}\times\mathfrak{B}.$$ This gives rise to the blow-up
	$$\widetilde{\mathfrak{X}\times\mathfrak{B}}\coloneqq\mathrm{Bl}_{\Gamma}(\mathfrak{X}\times\mathfrak{B})\overset{\pi}\longrightarrow\mathfrak{X}\times\mathfrak{B}$$ of $\mathfrak{X}\times\mathfrak{B}$ along $\Gamma$, as well as the exceptional divisor
	$$\mathcal{E}_{\Gamma}(\mathfrak{X}\times\mathfrak{B})\subseteq\widetilde{\mathfrak{X}\times\mathfrak{B}}.$$ We write $\theta_{\mathfrak{X}}:\mathfrak{X}\times\mathfrak{B}\longrightarrow\mathfrak{X}$ and $\theta_{\mathfrak{B}}:\mathfrak{X}\times\mathfrak{B}\longrightarrow\mathfrak{B}$ for the usual projections, and $\pi_{\mathfrak{X}}:\widetilde{\mathfrak{X}\times\mathfrak{B}}\longrightarrow\mathfrak{X}$ and $\pi_{\mathfrak{B}}:\widetilde{\mathfrak{X}\times\mathfrak{B}}\longrightarrow\mathfrak{B}$ for the results of composing $\pi$ with these projections.
	
	In what follows, $\mathfrak{X}\times\mathfrak{B}$ carries the product of the Poisson structure on $\mathfrak{X}$ and zero Poisson structure on $\mathfrak{B}$. This is encoded in a Poisson bracket on the coordinate algebra $\mathcal{A}\otimes_{\mathbb{C}}\mathcal{Z}(\mathcal{A})$ of $\mathfrak{X}\times\mathfrak{B}$. A straightforward exercise shows that this bracket satisfies
	$$\{a_1\otimes b_1,a_2\otimes b_2\}=\{a_1,a_2\}\otimes b_1b_2$$ for all $a_1,a_2\in\mathcal{A}$ and $b_1,b_2\in\mathcal{Z}(\mathcal{A})$. We will show that this blow-up satisfies Polishchuk's criterion discussed in Definition~\ref{Definition: Polishchuk}.
	
	\begin{theorem}\label{Theorem:IntegrableSystem}
		With the setup above, the following statements are true.
		\begin{itemize}
			\item[\textup{(i)}] The pair $(\mathfrak{X}\times\mathfrak{B},\Gamma)$ satisfies Polishchuk's criterion. In particular, $\widetilde{\mathfrak{X}\times\mathfrak{B}}$ is Poisson and $\mathcal{E}_{\Gamma}(\mathfrak{X}\times\mathfrak{B})\subseteq\widetilde{\mathfrak{X}\times\mathfrak{B}}$ is a Poisson subscheme.
			\item[\textup{(ii)}] Assume that $\mathcal{Z}(\mathcal{A})$ is finitely generated as an associative algebra. For each closed point $b\in\mathfrak{B}$, there is a canonical scheme isomorphism $\pi_{\mathfrak{B}}^{-1}(b)\cong\mathrm{Bl}_{\tau^{-1}(b)}(\mathfrak{X})$.
			\item[\textup{(iii)}] Retain the setup of \textup{(ii)}, while equipping $\widetilde{\mathfrak{X}\times\mathfrak{B}}$ and $\mathrm{Bl}_{\tau^{-1}(b)}(\mathfrak{X})$ with the Poisson structures from \textup{(i)} and Corollary \ref{Corollary: Poisson structure on blow-up}, respectively. Then the composite morphism \begin{equation}\label{Equation: Composite}\mathrm{Bl}_{\tau^{-1}(b)}(\mathfrak{X})\overset{\cong}\longrightarrow\pi_{\mathfrak{B}}^{-1}(b)\longhookrightarrow\widetilde{\mathfrak{X}\times\mathfrak{B}}\end{equation} is Poisson.
		\end{itemize}
	\end{theorem}
	
	\begin{proof}
		Let $\mathcal{I}\subseteq\mathcal{A}\otimes_{\mathbb{C}}\mathcal{Z}(\mathcal{A})$ denote the ideal of $\Gamma$ in $\mathfrak{X}\times\mathfrak{B}$. By Theorem \ref{Theorem: Polishchuk}, Part \textup{(i)} would follow once we show that $\mathcal{I}$ is Poisson and strongly coisotropic. We first observe that
		\begin{equation}\label{Equation: Generators1}\mathcal{I}=\langle f\otimes 1-1\otimes f:f\in\mathcal{Z}(\mathcal{A})\rangle.\end{equation} In particular, $\mathcal{I}$ is generated by a subset of the Poisson center $\mathcal{Z}(\mathcal{A}\otimes_{\mathbb{C}}\mathcal{Z}(\mathcal{A}))\subseteq \mathcal{A}\otimes_{\mathbb{C}}\mathcal{Z}(\mathcal{A})$. Proposition \ref{Proposition: Computation} now implies that $\mathcal{I}$ is Poisson and strongly coisotropic.
		
		We next prove \textup{(ii)}. Consider the ideal $\mathcal{J}\subseteq\mathcal{Z}(\mathcal{A})$ of $\{b\}$ in $\mathfrak{B}$, noting that $\mathcal{A}\otimes_{\mathbb{C}}\mathcal{J}\subseteq\mathcal{A}\otimes_{\mathbb{C}}\mathcal{Z}(\mathcal{A})$ is the ideal of $\theta_{\mathfrak{B}}^{-1}(b)$ in $\mathfrak{X}\times\mathfrak{B}$. At the same time, $\pi_{\mathfrak{B}}^{-1}(b)$ is the scheme-theoretic preimage $\pi^{-1}(\theta_{\mathfrak{B}}^{-1}(b))$. It follows that $\pi_{\mathfrak{B}}^{-1}(b)$ is canonically isomorphic to the blow-up of $\theta_{\mathfrak{B}}^{-1}(b)$ along the scheme-theoretic intersection $\Gamma\cap\theta_{\mathfrak{B}}^{-1}(b)$ \cite[Lemma 22.2.7]{vak:25}. The ideal of this intersection in $\theta_{\mathfrak{B}}^{-1}(b)$ is $(\mathcal{I}+\mathcal{A}\otimes_{\mathbb{C}}\mathcal{J})/(\mathcal{A}\otimes_{\mathbb{C}}\mathcal{J})\subseteq(\mathcal{A}\otimes_{\mathbb{C}}\mathcal{Z}(\mathcal{A}))/(\mathcal{A}\otimes_{\mathbb{C}}\mathcal{J})$, and the ideal of $\tau^{-1}(b)$ in $\mathfrak{X}$ is $\langle \mathcal{J}\rangle\subseteq\mathcal{A}$. It therefore suffices to exhibit an algebra isomorphism $\mathcal{A}\overset{\cong}\longrightarrow(\mathcal{A}\otimes_{\mathbb{C}}\mathcal{Z}(\mathcal{A}))/(\mathcal{A}\otimes_{\mathbb{C}}\mathcal{J})$ that identifies $\langle\mathcal{J}\rangle$ with $(\mathcal{I}+\mathcal{A}\otimes_{\mathbb{C}}\mathcal{J})/(\mathcal{A}\otimes_{\mathbb{C}}\mathcal{J})$.
		
		Since $\mathcal{Z}(\mathcal{A})$ is a finitely generated and $b$ is a closed point in $\mathfrak{B}$, we have $$(\mathcal{A}\otimes_{\mathbb{C}}\mathcal{Z}(\mathcal{A}))/(\mathcal{A}\otimes_{\mathbb{C}}\mathcal{J})\cong\mathcal{A}\otimes_{\mathbb{C}}(\mathcal{Z}(\mathcal{A})/\mathcal{J})\cong\mathcal{A}\otimes_{\mathbb{C}}\mathbb{C}\cong\mathcal{A}.$$ By composing these isomorphisms, we obtain a Poisson algebra isomorphism $$\varphi:\mathcal{A}\overset{\cong}\longrightarrow(\mathcal{A}\otimes_{\mathbb{C}}\mathcal{Z}(\mathcal{A}))/(\mathcal{A}\otimes_{\mathbb{C}}\mathcal{J}),\quad a\mapsto[a\otimes 1].$$ It remains only to prove that $\varphi(\langle\mathcal{J}\rangle)=(\mathcal{I}+\mathcal{A}\otimes_{\mathbb{C}}\mathcal{J})/(\mathcal{A}\otimes_{\mathbb{C}}\mathcal{J})$. We observe that $\mathcal{A}\otimes_{\mathbb{C}}\mathcal{J}=\langle 1\otimes f:f\in\mathcal{J}\rangle$, which together with Equation~\eqref{Equation: Generators1} implies that
		$$\mathcal{I}+\mathcal{A}\otimes_{\mathbb{C}}\mathcal{J}=\langle f\otimes 1 - 1\otimes f:f\in\mathcal{J}\rangle+\langle 1\otimes f:f\in\mathcal{J}\rangle=\langle f\otimes 1:f\in\mathcal{J}\rangle+\mathcal{A}\otimes_{\mathbb{C}}\mathcal{J}.$$ We conclude that $$(\mathcal{I}+\mathcal{A}\otimes_{\mathbb{C}}\mathcal{J})/(\mathcal{A}\otimes_{\mathbb{C}}\mathcal{J})=\langle[f\otimes 1]:f\in\mathcal{J}\rangle=\varphi(\langle J\rangle).$$
		
		Finally, we prove Part~\textup{(iii)}. Our proof of \textup{(ii)} implies that $\theta_{\mathfrak{B}}^{-1}(b)$ is a Poisson subscheme of $\mathfrak{X}\times\mathfrak{B}$, and that the natural isomorphism $\mathfrak{X}\overset{\cong}\longrightarrow\theta_{\mathfrak{B}}^{-1}(b)$ is one of Poisson schemes. The same proof implies that this isomorphism identifies the subschemes $\tau^{-1}(b)\subseteq\mathfrak{X}$ and $\Gamma\cap\theta_{\mathfrak{B}}^{-1}(b)\subseteq\theta_{\mathfrak{B}}^{-1}(b)$. We thereby obtain a Poisson scheme isomorphism
		$$\mathrm{Bl}_{\tau^{-1}(b)}(\mathfrak{X})\overset{\cong}\longrightarrow\mathrm{Bl}_{\Gamma\cap\theta_{\mathfrak{B}}^{-1}(b)}(\theta_{\mathfrak{B}}^{-1}(b)),$$
		where both Poisson structures are defined via Theorem \ref{Theorem: Polishchuk}. The image of \eqref{Equation: Composite} is $\mathrm{Bl}_{\Gamma\cap\theta_{\mathfrak{B}}^{-1}(b)}(\theta_{\mathfrak{B}}^{-1}(b))$, viewed as a closed subscheme of $\widetilde{\mathfrak{X}\times\mathfrak{B}}$ via \cite[Lemma 22.2.7]{vak:25}. It therefore suffices to prove that the closed embedding of $\mathrm{Bl}_{\Gamma\cap\theta_{\mathfrak{B}}^{-1}(b)}(\theta_{\mathfrak{B}}^{-1}(b))$ into $\widetilde{\mathfrak{X}\times\mathfrak{B}}$ is Poisson.
		
		The quotient map $$\mathcal{A}\otimes_{\mathbb{C}}\mathcal{Z}(\mathcal{A})\longrightarrow(\mathcal{A}\otimes_{\mathbb{C}}\mathcal{Z}(\mathcal{A}))/(\mathcal{A}\otimes_{\mathbb{C}}\mathcal{J})$$ induces a surjective morphism $$\psi:\mathrm{Rees}_{\mathcal{I}}(\mathcal{A}\otimes_{\mathbb{C}}\mathcal{Z}(\mathcal{A}))\longrightarrow\mathrm{Rees}_{(\mathcal{I}+\mathcal{A}\otimes_{\mathbb{C}}\mathcal{J})/(\mathcal{A}\otimes_{\mathbb{C}}\mathcal{J})}((\mathcal{A}\otimes_{\mathbb{C}}\mathcal{Z}(\mathcal{A}))/(\mathcal{A}\otimes_{\mathbb{C}}\mathcal{J}))$$ of $\mathbb{Z}_{\geq 0}$--graded algebras. By applying Proj to $\psi$, we obtain the closed embedding of $\mathrm{Bl}_{\Gamma\cap\theta_{\mathfrak{B}}^{-1}(b)}(\theta_{\mathfrak{B}}^{-1}(b))$ in $\widetilde{\mathfrak{X}\times\mathfrak{B}}$ referenced above. On the other hand, it is clear that $\psi$ induces surjective Poisson algebra morphisms on the algebras corresponding to the affine open sets described in Remark~\ref{Remark: Affine open}. This combines with Polishchuk's proof of Theorem \ref{Theorem: Polishchuk} to imply that the closed embedding of $\mathrm{Bl}_{\Gamma\cap\theta_{\mathfrak{B}}^{-1}(b)}(\theta_{\mathfrak{B}}^{-1}(b))$ into $\widetilde{\mathfrak{X}\times\mathfrak{B}}$ is Poisson.
	\end{proof}
	
	\section{The adjoint quotient}
	
	We now specialize the constructions in Section~\ref{Section: Poisson-geometric results} to the case where $\mathcal{A}=\mathrm{S}(\g)=\mathrm{S}(\g^*)$ for a finite-dimensional complex semisimple Lie algebra $\g$. This specialization involves properties of the adjoint quotient map $\g\longrightarrow\mathrm{Spec}(\mathrm{S}(\g^*)^G)\eqqcolon\c$, where $G$ is a connected complex semisimple group with Lie algebra $\g$. In Section \ref{Subsection: Lie-theoretic setup}, we review Kostant's results on the adjoint quotient. This leads to Section \ref{Subsection: Regular orbits}, where we discuss blow-ups of $\mathfrak{g}$ along closures of regular adjoint orbits. Section \ref{Subsection: Regular semisimple orbits} then specializes to regular semisimple orbits, yielding the symplectic leaves of the corresponding blow-ups. We obtain results on the exceptional divisor of the blow-up along a regular adjoint orbit closure in Section \ref{Subsection: The exceptional divisor}. In Section \ref{Subsection: Relative version}, we study the Poisson family $\widetilde{\g\times\c}\longrightarrow\c$ of blow-ups along regular adjoint orbit closures. This family is shown to be flat and conical, and its total space $\widetilde{\g\times\c}$ is given a Hamiltonian $G$-variety structure. In Section~\ref{Subsection: Quantization}, we discuss the structure and Poisson bracket on the Rees algebra encoding $\widetilde{\g\times\c}$, with a view towards quantization.
	
	\subsection{Lie-theoretic setup}\label{Subsection: Lie-theoretic setup}
	Let $\g$ be a finite-dimensional complex semisimple Lie algebra of rank $\ell$, and $G$ be a connected complex semisimple group integrating $\g$. Consider the adjoint representations
	$$\mathrm{Ad}:G\longrightarrow\mathrm{GL}(\g)\quad\text{and}\quad\mathrm{ad}:\g\longrightarrow\mathfrak{gl}(\g)$$ of $G$ and $\g$, respectively. Write $G_x\subseteq G$ and $\g_x\coloneqq\mathrm{ker}(\ad_x)\subseteq\g$ for the $G$-stabilizer and $\g$-centralizer of $x\in\g$; then $\g_x$ is the Lie algebra of $G_x$. An element $x\in\g$ is called \textit{regular} if $\dim(\g_x)=\ell$, and we write $\g_{\text{reg}}\subseteq\g$ for the $G$-invariant, open, dense subset of regular elements. One calls $x\in\g$ \textit{semisimple} (resp. \textit{nilpotent}) if the endomorphism $\mathrm{ad}_x:\g\longrightarrow\g$ is diagonalizable (resp. nilpotent). An adjoint orbit $\mathcal{O}\subseteq\g$ is called regular (resp. semisimple, nilpotent) if it contains a regular (resp. semisimple, nilpotent) element. An adjoint orbit $\mathcal{O}$ is semisimple if and only if $\overline{\mathcal{O}}=\mathcal{O}$, and nilpotent if and only if $0\in\overline{\mathcal{O}}$.
	
	The symmetric algebra $\mathrm{S}(\mathfrak{g})$ carries a unique Poisson bracket satisfying $\{x,y\}=[x,y]$ for all $x,y\in\mathfrak{g}$. On the other hand, the Killing form induces a $G$-module isomorphism $\g^*\cong\g$. The Poisson bracket on $\mathrm{S}(\g)$ thereby determines one on $\mathrm{S}(\g^*)$, rendering $\g$ an affine Poisson variety with symplectic leaves given by the adjoint $G$-orbits. It follows that the above-defined subset $\g_{\text{reg}}\subseteq\g$ coincides with the regular locus of $\g$ as a Poisson variety. Moreover, the Poisson center $\mathcal{Z}(\mathrm{S}(\g^*))=\mathrm{S}(\g^*)^G$ is a polynomial algebra on $\ell$ indeterminates. 
	
	The \textit{adjoint quotient} of $\mathfrak{g}$ is the affine Poisson variety morphism induced by the inclusion $\mathrm{S}(\g^*)^G\subseteq\mathrm{S}(\g^*)$:
	$$\g=\mathrm{Spec}(\mathrm{S}(\g^*))\overset{\tau}\longrightarrow\mathrm{Spec}(\mathrm{S}(\g^*)^G)=\g\sll{}G\eqqcolon\mathfrak{c},$$
	where $\mathfrak{c}$ carries the zero Poisson structure. It has rich properties in geometric representation theory, and its study was initiated by Kostant in \cite{kos:63}. The fibers of $\tau$ are precisely the closures of the regular adjoint orbits in $\g$ \cite[Theorem 3]{kos:63}, and so are irreducible. Each fiber is also known to be reduced \cite[Lemma 4]{kos:63} and a complete intersection \cite[Theorem 0.7]{kos:63}.
	
	\subsection{Regular orbits}\label{Subsection: Regular orbits}
	Let $\mathcal{O}\subseteq\g$ be a regular adjoint orbit. Note that the defining ideal of $\overline{\mathcal{O}}$ in $\g$ is $\mathcal{I}_{\overline{\mathcal{O}}}=\langle f_1-\alpha_1,\ldots,f_{\ell}-\alpha_{\ell}\rangle$ for a unique $(\alpha_1,\ldots,\alpha_{\ell})\in\mathbb{C}^{\ell}$, where $f_1,\ldots,f_{\ell}$ are homogeneous generators of $\mathrm{S}(\g^*)^G$ that form a regular sequence in $\mathrm{S}(\g^*)$. Recall that for a commutative unital algebra $A$ over $\mathbb{C}$, an ordered tuple $(a_1,\ldots,a_k)\in A^k$ is called a \textit{regular sequence} if $a_i$ is not a zero-divisor in $A/\langle a_1,\ldots,a_{i-1}\rangle$ for all $i\in\{1,\ldots,k\}$. Our next result is that $f_1-\alpha_1,\ldots,f_{\ell}-\alpha_{\ell}$ form a regular sequence in $\mathrm{S}(\g^*)$. 
	
	\begin{lemma}\label{Lemma: Regular}
		Let $f_1,\ldots,f_{\ell}\in\mathrm{S}(\g^*)^G$ be as above. For all $\alpha_1,\ldots,\alpha_{\ell}\in\mathbb{C}$, $(f_1-\alpha_1,\ldots,f_{\ell}-\alpha_{\ell})$ is a regular sequence in $\mathrm{S}(\g^*)$.
	\end{lemma}
	
	\begin{proof}
		By \cite[Theorem 11]{kos:63}, $\mathrm{S}(\g^*)$ is free as a module over $\mathrm{S}(\g^*)^G$. This implies that $(f_{\sigma(1)},\ldots,f_{\sigma(\ell)})$ is a regular sequence in $\mathrm{S}(\g^*)$ for all $\sigma\in S_{\ell}$ \cite[Proposition II.2.1]{pan:06}. An application of \cite[Corollary 5.3]{hem} now reveals that $(f_1-\alpha_1,\ldots,f_{\ell}-\alpha_{\ell})$ is a regular sequence in $\mathrm{S}(\g^*)$ for all $\alpha_1,\ldots,\alpha_{\ell}\in\mathbb{C}$.
	\end{proof}

	Since $\overline{\mathcal{O}}$ is a fiber of $\tau$, Corollary \ref{Corollary: Poisson structure on blow-up} yields a unique Poisson variety structure on $\mathrm{Bl}_{\overline{\mathcal{O}}}(\g)$ for which the blow-up morphism $\pi:\mathrm{Bl}_{\overline{\mathcal{O}}}(\g)\longrightarrow\g$ is Poisson. The same corollary tells us that $\mathcal{E}_{\overline{\mathcal{O}}}(\g)\coloneqq\pi^{-1}(\overline{\mathcal{O}})\subseteq\mathrm{Bl}_{\overline{\mathcal{O}}}(\g)$ is a Poisson subvariety. To expand on these facts, consider the morphism $$\phi:\g\setminus\overline{\mathcal{O}}\longrightarrow\mathbb{P}^{\ell-1},\quad x\mapsto [f_1(x)-\alpha_1:\cdots:f_{\ell}(x)-\alpha_{\ell}]$$ and its graph $\Gamma_{\phi}\subseteq\g\times\mathbb{P}^{\ell-1}$. Recall that $\mathrm{Bl}_{\overline{\mathcal{O}}}(\g)$ is the closure $\overline{\Gamma_{\phi}}\subseteq\g\times\mathbb{P}^{\ell-1}$, and that the blow-up morphism $\pi:\mathrm{Bl}_{\overline{\mathcal{O}}}(\g)\longrightarrow\g$ is obtained by restricting the projection $\g\times\mathbb{P}^{\ell-1}\longrightarrow\g$ to $\mathrm{Bl}_{\overline{\mathcal{O}}}(\g)$ \cite[Section 4]{hau}. Using Lemma \ref{Lemma: Regular} and \cite[Section 4]{hau}, we obtain 
	\begin{equation}\label{Equation:Blowup}\mathrm{Bl}_{\overline{\mathcal{O}}}(\g)=\{(x,[z_1:\cdots: z_{\ell}])\in\g\times\mathbb{P}^{\ell-1}:(f_i(x)-\alpha_i)z_j=(f_j(x)-\alpha_j)z_i\text{ for all }i,j\in\{1,\ldots,\ell\}\}.\end{equation}
	Since $\phi$ is $G$-invariant, $\mathrm{Bl}_{\overline{\mathcal{O}}}(\g)$ is invariant under the $G$-action on $\g\times\mathbb{P}^{\ell-1}$ given by $g\cdot (x,[v])=(\mathrm{Ad}_g(x),[v])$ for all $g\in G$ and $(x,[v])\in\g\times\mathbb{P}^{\ell-1}$. Note that $\pi$ is equivariant with respect to this $G$-action on $\mathrm{Bl}_{\overline{\mathcal{O}}}(\g)$. The action also leaves $\mathcal{E}_{\overline{\mathcal{O}}}(\g)$ invariant, so that $\pi$ restricts to a $G$-equivariant Poisson variety isomorphism $\mathrm{Bl}_{\overline{\mathcal{O}}}(\g)\setminus\mathcal{E}_{\overline{\mathcal{O}}}(\g)\longrightarrow\g\setminus\overline{\mathcal{O}}$. We conclude that $\mathrm{Bl}_{\overline{\mathcal{O}}}(\g)\setminus\mathcal{E}_{\overline{\mathcal{O}}}(\g)$ is a Poisson Hamiltonian $G$-variety with moment map $$\pi\big\vert_{\mathrm{Bl}_{\overline{\mathcal{O}}}(\g)\setminus\mathcal{E}_{\overline{\mathcal{O}}}(\g)}:\mathrm{Bl}_{\overline{\mathcal{O}}}(\g)\setminus\mathcal{E}_{\overline{\mathcal{O}}}(\g)\longrightarrow\g.$$ 
	If $\mathcal{O}$ is semisimple, then $\mathrm{Bl}_{\overline{\mathcal{O}}}(\g)=\mathrm{Bl}_{\mathcal{O}}(\g)$ is smooth. The fact that $\mathrm{Bl}_{\overline{\mathcal{O}}}(\g)\setminus\mathcal{E}_{\overline{\mathcal{O}}}(\g)$ is open and dense in $\mathrm{Bl}_{\overline{\mathcal{O}}}(\g)$ then implies that the latter is a smooth Poisson Hamiltonian $G$-variety with moment map $\pi$.

	\subsection{Regular semisimple orbits}\label{Subsection: Regular semisimple orbits}
	
	Suppose that $\mathcal{O}$ is a regular semisimple adjoint orbit. It follows that $\overline{\mathcal{O}}=\mathcal{O}$ is smooth, so that $\mathrm{Bl}_{\overline{\mathcal{O}}}(\g)=\mathrm{Bl}_{\mathcal{O}}(\g)$ and $\mathcal{E}_{\overline{\mathcal{O}}}(\g)=\mathcal{E}_{\mathcal{O}}(\g)$ are smooth as well. Applying Proposition \ref{Proposition: Poisson expansion} and letting $\pi:\mathrm{Bl}_{\mathcal{O}}(\g)\longrightarrow\g$ denote the blow-up morphism, we get that $\mathrm{Prk}(\mathrm{Bl}_{\mathcal{O}}(\g))=\mathrm{Prk}(\g)=\dim(\g)-\ell$ and $$\mathrm{Bl}_{\mathcal{O}}(\g)_{\text{reg}}=\pi^{-1}(\g_{\text{reg}})=\pi^{-1}(\g_{\text{reg}}\setminus\mathcal{O})\cup\mathcal{E}_{\mathcal{O}}(\g).$$ In the following result, we identify the symplectic leaves and $G$-orbits in $\mathrm{Bl}_{\mathcal{O}}(\g)$.
	
	\begin{proposition}\label{Proposition: Orbit coincidence}
		If $\mathcal{O}\subseteq\g$ is a regular semisimple adjoint orbit, then the symplectic leaves of $\mathrm{Bl}_{\mathcal{O}}(\g)$ coincide with the $G$-orbits in $\mathrm{Bl}_{\mathcal{O}}(\g)$. 
	\end{proposition}
	
	\begin{proof}
		Noting that $\mathrm{Bl}_{\mathcal{O}}(\g)$ is smooth, it suffices to prove that the symplectic distribution of $\mathrm{Bl}_{\mathcal{O}}(\g)$ coincides with the distribution induced by generating vector fields for the $G$-action on $\mathrm{Bl}_{\mathcal{O}}(\g)$. Our task is therefore to prove the following for all $x\in\mathrm{Bl}_{\mathcal{O}}(\g)$: $T_x(L_x)\subseteq T_x(\mathrm{Bl}_{\mathcal{O}}(\g))$ is the subspace of generating vector fields at $x$ for the $G$-action on $\mathrm{Bl}_{\mathcal{O}}(\g)$. We first assume that $x\in\mathrm{Bl}_{\mathcal{O}}(\g)\setminus\mathcal{E}_{\mathcal{O}}(\g)$. Note that $T_{\pi(x)}(L_{\pi(x)})\subseteq\g$ is the subspace of generating vector fields at $\pi(x)$ for the adjoint action. Since $\pi$ restricts to a $G$-equivariant Poisson isomorphism $\mathrm{Bl}_{\mathcal{O}}(\g)\setminus\mathcal{E}_{\mathcal{O}}(\g)\longrightarrow\g\setminus\mathcal{O}$, $T_x(L_x)$ is as advertised in the statement of our task. It therefore suffices to assume that $x\in\mathcal{E}_{\mathcal{O}}(\g)$. As before, $T_{\pi(x)}(L_{\pi(x)})=T_{\pi(x)}\mathcal{O}\subseteq\g$ is the subspace of generating vector fields at $\pi(x)$ for the adjoint action. Proposition \ref{Proposition: Poisson expansion} also tells us that $\dim(L_x)=\dim(L_{\pi(x)})$. Using the $G$-equivariance of $\pi$, a proof similar to that of Lemma \ref{Lemma: Inequality} implies that $T_x(L_x)$ is as advertised. This completes the proof.
	\end{proof}

	\subsection{The exceptional divisor}\label{Subsection: The exceptional divisor} Let $\mathcal{O}\subseteq\g$ be a regular adjoint orbit. We use the description of $\overline{\mathcal{O}}$ in terms of a regular sequence to characterize the exceptional divisor $\mathcal{E}_{\overline{\mathcal{O}}}(\g)$.

	\begin{proposition}\label{Proposition: Regular orbit}
		If $\mathcal{O}\subseteq\g$ is a regular adjoint orbit, then $\mathcal{E}_{\overline{\mathcal{O}}}(\g)\cong\overline{\mathcal{O}}\times\mathbb{P}^{\ell-1}$. 
	\end{proposition}
	
	\begin{proof}
		Let $(\alpha_1,\ldots,\alpha_{\ell})\in\mathbb{C}^{\ell}$ be such that $\overline{\mathcal{O}}=(f_1,\ldots,f_{\ell})^{-1}(\alpha_1,\ldots,\alpha_{\ell})$. This amounts to the ideal of $\overline{\mathcal{O}}$ in $\g$ being generated by $f_1-\alpha_1,\ldots,f_{\ell}-\alpha_{\ell}$. Lemma \ref{Lemma: Regular} implies that these generators form a regular sequence in $\mathrm{S}(\g^*)$. Using the description of $\mathrm{Bl}_{\overline{\mathcal{O}}}(\g)$ in Equation~\ref{Equation:Blowup}, it follows that $\mathcal{E}_{\overline{\mathcal{O}}}(\g)=\pi^{-1}(\overline{\mathcal{O}})=\overline{\mathcal{O}}\times\mathbb{P}^{\ell-1}$.
	\end{proof}
	
	A more geometrically conceptual approach to capturing some of the structure of $\mathcal{E}_{\overline{\mathcal{O}}}(\g)$ is by means of normal bundles. More precisely, one can proceed as follows. The $G$-action on $\mathcal{O}$ induces one on $(T\g)\big\vert_{\mathcal{O}}$ by vector bundle automorphisms. This action preserves $T\mathcal{O}$, and so yields a $G$-action on the normal bundle $\nu_{\mathcal{O}}\coloneqq(T\g)\big\vert_{\mathcal{O}}/T\mathcal{O}$ by vector bundle automorphisms. At the same time, the fiber of $\nu_{\mathcal{O}}$ at $x\in\mathcal{O}$ is the $G_x$-module quotient $\g/[\g,x]$. This allows one to write
	$$\nu_{\mathcal{O}}=\{(x,[\xi]):x\in\mathcal{O},\text{ }[\xi]\in\g/[\g,x]\},$$ after which the bundle map $\nu_{\mathcal{O}}\longrightarrow\mathcal{O}$ becomes projection to the first factor. The action of $G$ on $\nu_{\mathcal{O}}$ is then given by
	$$g\cdot(x,[\xi])\coloneqq(\mathrm{Ad}_g(x),[\mathrm{Ad}_g(\xi)]),\quad g\in G,\text{ }(x,[\xi])\in \nu_{\mathcal{O}}.$$ 
	
	Given $x\in\mathcal{O}$, consider the $G$-equivariant vector bundle $G\times^{G_x}(\g/[\g,x])\longrightarrow G/G_x$, associated to the $G_x$-module $\g/[\g,x]$. Using the $G$-variety isomorphism
	$$G/G_x\overset{\cong}\longrightarrow\mathcal{O},\quad [g]\mapsto\mathrm{Ad}_g(x),$$ we can regard $G\times^{G_x}(\g/[\g,x])$ as a $G$-equivariant vector bundle over $\mathcal{O}$. It follows that \begin{equation}\label{Equation: Trivialization}G\times^{G_x}(\g/[\g,x])\overset{\cong}\longrightarrow \nu_{\mathcal{O}},\quad[g:[\xi]]\mapsto(\mathrm{Ad}_g(x),[\mathrm{Ad}_g(\xi)])\end{equation} is a $G$-equivariant isomorphism of vector bundles over $G/G_x\cong\mathcal{O}$. On the other hand, recall from Section \ref{Subsection: Regular orbits} that $\mathrm{Bl}_{\overline{\mathcal{O}}}(\g)$ carries a $G$-action for which the blow-up morphism $\pi:\mathrm{Bl}_{\overline{\mathcal{O}}}(\g)\longrightarrow\g$ is $G$-equivariant. These considerations lead to the following result.

	\begin{proposition}\label{Proposition: Trivial bundle}
		Suppose that $\mathcal{O}\subseteq\g$ is a regular adjoint orbit. 
		\begin{itemize}
			\item[\textup{(i)}] The normal bundle $\nu_{\mathcal{O}}$ is trivial as a $G$-equivariant vector bundle over $\mathcal{O}$.
			\item[\textup{(ii)}] There is a $G$-equivariant isomorphism $\pi^{-1}(\mathcal{O})\cong\mathcal{O}\times\mathbb{P}^{\ell-1}$ of varieties over $\mathcal{O}$, where $G$ acts on $\mathcal{O}\times\mathbb{P}^{\ell-1}$ via the adjoint action on the first factor.
		\end{itemize}
	\end{proposition}
	
	\begin{proof}
		To prove (i), choose $x\in\mathcal{O}$. Note that the dual $G_x$-module $(\g/[\g,x])^*$ is isomorphic to the annihilator of $[\g,x]$ in $\g^*$. The Killing form identifies the latter module with $[\g,x]^{\perp}=\g_x$. Since $G_x$ is abelian, $\g_x$ is a trivial $G_x$-module. This argument shows $(\g/[\g,x])^*$ to be a trivial $G_x$-module, implying that $\g/[\g,x]$ is trivial as well. It now follows from \eqref{Equation: Trivialization} that  
		$$\nu_{\mathcal{O}}\cong(G/G_x)\times(\g/[\g,x])$$ as $G$-equivariant vector bundles over $G/G_x\cong\mathcal{O}$. 
		
		We now verify (ii). Observe that $\dim(\g/[\g,x])=\dim(\g_x)=\ell$ for all $x\in\mathcal{O}$. Part (i) now implies that $\mathbb{P}(\nu_{\mathcal{O}})\cong\mathcal{O}\times\mathbb{P}^{\ell-1}$ as varieties over $\mathcal{O}$. It remains only to establish that $\mathbb{P}(\nu_{\mathcal{O}})$ and $\pi^{-1}(\mathcal{O})$ are $G$-equivariantly isomorphic as varieties over $\mathcal{O}$. A first step is to choose homogeneous, algebraically independent generators $f_1,\ldots,f_{\ell}$ of $\mathrm{S}(\g^*)^G$.  
		Consider the map $$f:\g\longrightarrow\mathbb{C}^{\ell},\quad x\mapsto (f_1(x),\ldots,f_{\ell}(x)).$$ For $x\in\g$, $\mathrm{rank}(\mathrm{d}f_x)=\ell$ if and only if $x\in\g_{\text{reg}}$. One also knows that $\overline{\mathcal{O}}\cap\g_{\text{reg}}=\mathcal{O}$ \cite[Theorem 3]{kos:63}. These considerations imply that $\mathcal{O}$ is the smooth locus of $\overline{\mathcal{O}}$, yielding the isomorphism
		$$\mathbb{P}(\nu_{\mathcal{O}})\overset{\cong}\longrightarrow\pi^{-1}(\mathcal{O}),\quad (x,[\xi])\mapsto (x,[\mathrm{d}f_x(\xi)])$$ of varieties over $\mathcal{O}$; see \cite[Section 22.3]{vak:25} and \cite[Appendix B.6]{ful}. Since $\mathrm{d}f_{\mathrm{Ad_g(x)}}=\mathrm{d}f_x\circ\mathrm{Ad}_{g^{-1}}$ for all $g\in G$ and $x\in\g$, this isomorphism is $G$-equivariant.
	\end{proof}

	\subsection{The flat conical family $\widetilde{\g\times\c}$}\label{Subsection: Relative version} We now examine the results of Section \ref{Subsection: A family} when $\mathcal{A}=\mathrm{S}(\g^*)$. In the notation of that section, we have that $\mathfrak{X}=\g$, $\mathfrak{B}=\mathrm{Spec}(\mathrm{S}(\g^*)^G)=\mathfrak{c}$, and $\tau:\mathfrak{X}\longrightarrow\mathfrak{B}$ is the adjoint quotient $\tau:\g\longrightarrow\mathfrak{c}$. The graph of $\tau$ is $\Gamma\coloneqq \g\times_{\mathfrak{c}}\mathfrak{c}\subseteq\g\times\mathfrak{c}$. Consider the blow-up of $\g\times\mathfrak{c}$ along this closed subscheme,
	$$\pi:\widetilde{\g\times\mathfrak{c}}\coloneqq\mathrm{Bl}_{\Gamma}(\g\times\mathfrak{c})\longrightarrow\g\times\mathfrak{c}.$$ Since $\g\cong\Gamma$ as schemes, $\widetilde{\g\times\mathfrak{c}}$ is the blow-up of a smooth variety along a smooth closed subscheme. It follows that $\widetilde{\g\times\mathfrak{c}}$ is smooth. Write $\theta_{\mathfrak{g}}:\mathfrak{g}\times\mathfrak{c}\longrightarrow\mathfrak{g}$ and $\theta_{\mathfrak{c}}:\mathfrak{g}\times\mathfrak{c}\longrightarrow\mathfrak{c}$ for the usual projections, and $$\pi_{\mathfrak{g}}:\widetilde{\mathfrak{g}\times\mathfrak{c}}\longrightarrow\mathfrak{g}\quad \text{and} \quad \pi_{\mathfrak{c}}:\widetilde{\mathfrak{g}\times\mathfrak{c}}\longrightarrow\mathfrak{c}$$ for the results of composing $\pi$ with $\theta_\g$ and $\theta_\c$ respectively. Let $G$ act on $\g\times\mathfrak{c}$ by $g\cdot(x,c)\coloneqq(\mathrm{Ad}_g(x),c)$ for all $g\in G$ and $(x,c)\in\g\times\mathfrak{c}$. This action is Hamiltonian, with moment map given by the projection $\theta_{\g}:\g\times\mathfrak{c}\longrightarrow\g$. 
	
	\begin{remark}\label{remark: Graph closure}
		As in the case of blowing up $\g$ along $\overline{\mathcal{O}}$ discussed in Section~\ref{Subsection: Regular orbits}, we can describe $\widetilde{\g\times\c}$ as a graph closure. 
		We fix homogeneous, algebraically independent generators $f_1,\ldots,f_{\ell}$ of $\mathrm{S}(\g^*)^G$ as before. This allows is to identify $\c=\mathbb{C}^{\ell}$ and $\Gamma=\g\times_{\mathbb{C}^{\ell}} \mathbb{C}^{\ell} \subseteq \g\times\mathbb{C}^{\ell}$. 	Consider the morphism $$\phi:(\g\times\mathbb{C}^{\ell})\setminus(\g\times_{\mathbb{C}^{\ell}}\mathbb{C}^{\ell})\longrightarrow\mathbb{P}^{\ell-1},\quad (x,(z_1,\ldots,z_{\ell}))\mapsto [f_1(x)-z_1:\cdots:f_{\ell}(x)-z_{\ell}]$$ and its graph $\mathrm{gr}(\phi)\subseteq\g\times\mathbb{C}^{\ell}\times\mathbb{P}^{\ell-1}$. Since the ideal of $\Gamma\subseteq\g\times\mathbb{C}^{\ell}$ is generated by $\{f_i\otimes 1-1\otimes x_i:i\in\{1,\ldots,\ell\}\}\subseteq\mathrm{S}(\g^*)\otimes_{\mathbb{C}}\mathbb{C}[x_1,\ldots,x_n]$, by \cite[Section 4]{hau} we can identify $$\widetilde{\g\times\mathfrak{c}}=\overline{\mathrm{gr}(\phi)}\subseteq\g\times\mathbb{C}^{\ell}\times\mathbb{P}^{\ell-1}.$$ The blow-up morphism $\pi$ is obtained by restricting the projection $\g\times\mathbb{C}^{\ell}\times\mathbb{P}^{\ell-1}\longrightarrow\g\times\mathbb{C}^{\ell}$ to $\widetilde{\g\times\mathfrak{c}}$.
	\end{remark}
	
	\begin{theorem}\label{Theorem: Main theorem Lie algebra}
		Endow $\g\times\mathfrak{c}$ with the Poisson Hamiltonian $G$-variety structure described above.
		\begin{itemize}
			\item[\textup{(i)}] The pair $(\g \times \mathfrak{c}, \Gamma)$ satisfies Polishchuk's criterion from Definition~\ref{Definition: Polishchuk}.  In particular, $\widetilde{\g \times \mathfrak{c}}$ is Poisson and $\mathcal{E}_{\Gamma}(\g
			\times \mathfrak{c})= \pi^{-1}(\Gamma)\subseteq\widetilde{\g \times \mathfrak{c}}$ is a Poisson subscheme.
			\item[\textup{(ii)}] For each point $c \in \mathfrak{c}$, there is a canonical scheme isomorphism $\pi^{-1}_{\mathfrak{c}}(c)\cong \mathrm{Bl}_{\tau^{-1}(c)}(\g)$. The following composite morphism is Poisson $$\mathrm{Bl}_{\tau^{-1}(c)}(\g)\overset{\cong}\longrightarrow\pi_{\mathfrak{c}}^{-1}(c)\longhookrightarrow\widetilde{\g \times \mathfrak{c}}.$$
			\item[\textup{(iii)}] The $G$-action on $\g\times\mathfrak{c}$ has a unique lift to a Poisson Hamiltonian $G$-variety structure on $\widetilde{\g\times\mathfrak{c}}$, and the exceptional divisor $\mathcal{E}_{\Gamma}(\g \times \mathfrak{c})\cong \g\times\mathbb{P}^{\ell-1}$ is $G$-invariant. The corresponding moment map is $\pi_{\g}:\widetilde{\g\times\mathfrak{c}}\longrightarrow\g$.
		\end{itemize}
	\end{theorem}
	
	\begin{proof}
		Parts (i) and (ii) are consequences of Theorem~\ref{Theorem:IntegrableSystem}. We now prove (iii). To this end, choose homogeneous, algebraically independent generators $f_1,\ldots,f_{\ell}$ of $\mathrm{S}(\g^*)^G$. We use these generators to freely identify $\mathfrak{c}$ with $\mathbb{C}^{\ell}$ in what follows.
		Consider the normal bundle $\nu\longrightarrow\Gamma$ of $\Gamma\subseteq\g\times\mathfrak{c}=\g\times\mathbb{C}^{\ell}$, as well as the map
		$$f:\g\times\mathbb{C}^{\ell}\longrightarrow\mathbb{C}^{\ell},\quad (x,(z_1,\ldots,z_{\ell}))\mapsto (f_1(x)-z_1,\ldots,f_{\ell}(x)-z_{\ell}).$$ Observe that $\Gamma=f^{-1}(0)$ scheme-theoretically. It follows that $$\mathbb{P}(\nu)\overset{\cong}\longrightarrow\mathcal{E}_{\Gamma}(\g \times \mathfrak{c})\subseteq \g\times\mathbb{C}^{\ell}\times\mathbb{P}^{\ell-1},\quad ((x,\tau(x)),[v])\mapsto(x,\tau(x),[\mathrm{d}f_x(v)])$$ defines an isomorphism of varieties over $\Gamma$; see \cite[Section 22.3]{vak:25} and \cite[Appendix B.6]{ful}. This isomorphism is $G$-equivariant, where the $G$-action on $\mathbb{P}(\nu)$ is induced by $\Gamma\subseteq\g\times\mathfrak{c}=\g\times\mathbb{C}^{\ell}$ being a smooth, $G$-invariant subvariety. It is also clear that 
		$$\g\times\mathbb{P}^{\ell-1}\overset{\cong}\longrightarrow\mathbb{P}(\nu),\quad (x,[z])\mapsto ((x,\tau(x)),[(0,z)])$$ defines a $G$-equivariant variety isomorphism, where $G$ acts on $\g\times\mathbb{P}^{\ell-1}$ via its action on the first factor.
		
		By the identification from Remark~\ref{remark: Graph closure}, we have $$\widetilde{\g\times\mathfrak{c}}=\overline{\mathrm{gr}(\phi)}\subseteq\g\times\mathbb{C}^{\ell}\times\mathbb{P}^{\ell-1}.$$ The moment map $\mu:\widetilde{\g\times\mathfrak{c}}\longrightarrow\g$ is thus obtained by restricting the projection $\g\times\mathbb{C}^{\ell}\times\mathbb{P}^{\ell-1}\longrightarrow\g$ to $\widetilde{\g\times\mathfrak{c}}$. Since $\phi$ is $G$-invariant, $\widetilde{\g\times\mathfrak{c}}$ is invariant under the following action of $G$ on $\g\times\mathbb{C}^{\ell}\times\mathbb{P}^{\ell-1}$: $g\cdot (x,v,[w])=(\mathrm{Ad}_g(x),v,[w])$ for all $g\in G$ and $(x,v,[w])\in\g\times\mathbb{C}^{\ell}\times\mathbb{P}^{\ell-1}$. Note that the blow-up morphism $\pi$ is equivariant with respect to this $G$-action on $\widetilde{\g\times\mathfrak{c}}$. The action is also observed to leave the divisor $\mathcal{E}_{\Gamma}(\g \times \mathfrak{c})\subseteq\widetilde{\g\times\mathfrak{c}}$ invariant, so that $\pi$ restricts to a $G$-equivariant Poisson variety isomorphism $\widetilde{\g\times\mathfrak{c}}\setminus\mathcal{E}_{\Gamma}(\g \times \mathfrak{c})\longrightarrow(\g\times\mathfrak{c})\setminus(\g\times_{\mathfrak{c}}\mathfrak{c})$. We conclude that $\widetilde{\g\times\mathfrak{c}}\setminus\mathcal{E}_{\Gamma}(\g \times \mathfrak{c})$ is a Poisson Hamiltonian $G$-variety with moment map $$\mu\big\vert_{\widetilde{\g\times\mathfrak{c}}\setminus\mathcal{E}_{\Gamma}(\g \times \mathfrak{c})}:\widetilde{\g\times\mathfrak{c}}\setminus\mathcal{E}_{\Gamma}(\g \times \mathfrak{c})\longrightarrow\g.$$ As $\widetilde{\g\times\mathfrak{c}}\setminus\mathcal{E}_{\Gamma}(\g \times \mathfrak{c})$ is open and dense in $\widetilde{\g\times\mathfrak{c}}$, it follows that $\widetilde{\g\times\mathfrak{c}}$ is a Poisson Hamiltonian $G$-variety with moment map $\mu$. This proves (iii).
	\end{proof}	
	
	We show next that this blow-up forms a flat family over $\c$.
	
	\begin{proposition}\label{Proposition: Actions}
		The morphism $\pi_{\mathfrak{c}}:\widetilde{\g\times\mathfrak{c}}\longrightarrow\mathfrak{c}$ is flat.
	\end{proposition}
	
	\begin{proof}
		We have that $\widetilde{\g\times\mathfrak{c}}$ and $\mathfrak{c}$ are both smooth. Therefore, by the principle of ``miracle flatness" \cite[Theorem 23.1]{mat:89}, it suffices to prove that all fibers of $\pi_{\mathfrak{c}}$ have the same dimension. This sufficient condition holds by virtue of Theorem \ref{Theorem: Main theorem Lie algebra}(ii) and the fact that $\dim(\mathrm{Bl}_{\tau^{-1}(c)}(\g))=\dim(\g)$ for all $c\in\mathfrak{c}$.
	\end{proof}
	
	We now examine the flat family $\pi_{\mathfrak{c}}:\widetilde{\g\times\mathfrak{c}}\longrightarrow\mathfrak{c}$ in more detail. To this end, let $\mathcal{N}\coloneqq \overline{\mathcal{O}_\text{reg}}\subseteq\g$ denote the nilpotent cone, where $\mathcal{O}_\text{reg}$ is the regular nilpotent orbit. We show that $\pi_{\mathfrak{c}}$ is conical in the following sense.
	\begin{proposition}\label{Proposition: Choice}
		There exist algebraic $\mathbb{C}^{\times}$-actions on $\mathfrak{c}$, $\g\times\mathfrak{c}$, and $\widetilde{\g\times\mathfrak{c}}$ with the following properties:
		\begin{itemize}
			\item[\textup{(i)}] $\Lim{t\rightarrow 0}(t\cdot c)=\mathfrak{o}$ for all $c\in\mathfrak{c}$, where $\mathfrak{o}\in\mathfrak{c}$ is the unique element with $\tau^{-1}(\mathfrak{o})=\mathcal{N}$;
			\item[\textup{(ii)}] the blow-up morphism $\pi:\widetilde{\g\times\mathfrak{c}}\longrightarrow\g\times\mathfrak{c}$ and projection $\theta_{\mathfrak{c}}:\g\times\mathfrak{c}\longrightarrow\mathfrak{c}$ are $\mathbb{C}^{\times}$-equivariant;
			\item[\textup{(iii)}] the subvariety $\Gamma\subseteq\g\times\mathfrak{c}$ and exceptional divisor $\mathcal{E}_{\Gamma}(\g\times\c)\subseteq\widetilde{\g\times\mathfrak{c}}$ are $\mathbb{C}^{\times}$-invariant. 
		\end{itemize}
	\end{proposition}
	
	\begin{proof}
		We start be recalling a $\mathbb{C}^\times$-action on $\c$ introduced in \cite{gan-gin:02}. Fix a principal $\mathfrak{sl}_2$-triple $(e,h,f)\in\g^{\times 3}$. Work of Kostant \cite{kos:63} implies that the adjoint quotient $\tau:\g\longrightarrow\mathfrak{c}$ restricts to a variety isomorphism \begin{equation}\label{Equation: Slice isomorphism} \tau\big\vert_{\mathcal{S}}:\mathcal{S}\overset{\cong}\longrightarrow\mathfrak{c},\end{equation} where $\mathcal{S}\coloneqq e+\g_f$ is the Kostant slice. There is a unique Lie algebra morphism $\phi:\mathfrak{sl}_2\longrightarrow\g$ satisfying
		$$\phi\left(\begin{bmatrix}0 & 1\\ 0 & 0\end{bmatrix}\right)=e,\quad \phi\left(\begin{bmatrix}1 & 0\\ 0 & -1\end{bmatrix}\right)=h,\quad\text{and}\quad \phi\left(\begin{bmatrix}0 & 0\\ 1 & 0\end{bmatrix}\right)=f.$$ Since $\operatorname{SL}_2$ is simply-connected, there is a unique algebraic group morphism $\varphi:\operatorname{SL}_2\longrightarrow G$ integrating $\phi$. This morphism has an associated cocharacter
		$$\lambda:\mathbb{C}^{\times}\longrightarrow G,\quad t\mapsto\varphi\left(\begin{bmatrix}t^{-1} & 0\\ 0 & t\end{bmatrix}\right).$$ Observe that
		$t\cdot x\coloneqq t^2\mathrm{Ad}_{\lambda(t)}(x)$ for $t\in\mathbb{C}^{\times}$ and $x\in \mathcal{S}$ is an algebraic action of $\mathbb{C}^{\times}$ on $\mathcal{S}$ satisfying \begin{equation}\label{Equation: Cone point}\Lim{t\rightarrow 0}(t\cdot x)=e\quad \forall \; x\in \mathcal{S}.\end{equation} 
		Equip $\mathfrak{c}$ with the algebraic $\mathbb{C}^{\times}$-action for which the isomorphism \eqref{Equation: Slice isomorphism} is equivariant. Since $\tau(e)=\mathfrak{o}$, \eqref{Equation: Cone point} implies that $\lim_{t\rightarrow 0}(t\cdot c)=\mathfrak{o}$ for all $c\in\mathfrak{c}$. This verifies (i).
		
		We now verify (ii). Consider the $\mathbb{C}^{\times}$-action on $\c$ is defined in the previous paragraph. Let $\mathbb{C}^{\times}$ act on $\g\times\c$ by
		\begin{equation}\label{Equation: Product action} t\cdot (x,c)\coloneqq (t^2x,t\cdot c),\quad t\in\mathbb{C}^{\times},\text{ }(x,c)\in\g\times\c.\end{equation} This action makes the projection $\theta_{\c}$ $\mathbb{C}^{\times}$-equivariant. To define the desired $\mathbb{C}^{\times}$-action on $\widetilde{\g\times\c}$, choose homogeneous, algebraically independent generators $f_1,\ldots,f_{\ell}$ of $\mathrm{S}(\g^*)^{G}$. These choices induce an identification $\c=\mathbb{C}^{\ell}$. Note that $\tau=(f_1,\ldots,f_{\ell}):\g\longrightarrow\mathbb{C}^{\ell}$, implying that \eqref{Equation: Slice isomorphism} is the isomorphism
		$$(f_1,\ldots,f_{\ell})\big\vert_{\mathcal{S}}:\mathcal{S}\overset{\cong}\longrightarrow\mathbb{C}^{\ell}.$$ The $\mathbb{C}^{\times}$-action on $\mathfrak{c}$ is now given by the following $\mathbb{C}^{\times}$-action on $\mathbb{C}^{\ell}$:
		\begin{equation}\label{Equation: New action}t\cdot(z_1,\ldots,z_{\ell})\coloneqq (t^{2d_1}z_1,\ldots,t^{2d_{\ell}}z_{\ell}),\quad t\in\mathbb{C}^{\times},\text{ }(z_1,\ldots,z_{\ell})\in\mathbb{C}^{\ell},\end{equation} where $d_1,\ldots,d_{\ell}\in\mathbb{Z}_{\geq 1}$ are the degrees of $f_1,\ldots f_{\ell}$, respectively. It follows that \eqref{Equation: Product action} is given by
		\begin{equation}\label{Equation: Funny action}t\cdot(x,(z_1,\ldots,z_{\ell}))= (t^2x,(t^{2d_1}z_1,\ldots,t^{2d_{\ell}}z_{\ell})),\quad t\in\mathbb{C}^{\times},\text{ }(x,(z_1,\ldots,z_{\ell}))\in\g\times\mathbb{C}^{\ell}.\end{equation}
		We must prove that \eqref{Equation: Funny action} lifts to an algebraic $\mathbb{C}^{\times}$-action on $\widetilde{\g\times\mathbb{C}^{\ell}}$ under $\pi:\widetilde{\g\times\mathbb{C}^{\ell}}\longrightarrow\g\times\mathbb{C}^{\ell}$. By the identification from Remark~\ref{remark: Graph closure}, we have
		$$\widetilde{\g\times\mathbb{C}^{\ell}}=\overline{\mathrm{gr}(\phi)}\subseteq\g\times\mathbb{C}^{\ell}\times\mathbb{P}^{\ell-1}.$$ Consider the $\mathbb{C}^{\times}$-action on $\g\times\mathbb{C}^{\ell}\times\mathbb{P}^{\ell-1}$ given by 
		$$t\cdot (x,(z_1,\ldots,z_{\ell}),[v])\coloneqq (t^2x,(t^{2d_1}z_1,\ldots,t^{2d_{\ell}}z_{\ell}),[v]),\quad t\in\mathbb{C}^{\times},\text{ }(x,(z_1,\ldots,z_{\ell}),[v])\in\g\times\mathbb{C}^{\ell}\times\mathbb{P}^{\ell-1}.$$ This action clearly makes the projection $\g\times\mathbb{C}^{\ell}\times\mathbb{P}^{\ell-1}\longrightarrow\g\times\mathbb{C}^{\ell}$ equivariant, where $\mathbb{C}^{\times}$ acts on $\g\times\mathbb{C}^{\ell}$ as described in \eqref{Equation: Funny action}. We also observe that $\mathrm{gr}(\phi)$ is preserved by this action, implying that $\overline{\mathrm{gr}(\phi)}\subseteq\widetilde{\g\times\mathbb{C}^{\ell}}$ is preserved as well. As $\pi$ results from restricting the projection $\g\times\mathbb{C}^{\ell}\times\mathbb{P}^{\ell-1}\longrightarrow\g\times\mathbb{C}^{\ell}$ to $\widetilde{\g\times\mathbb{C}^{\ell}}$, it is $\mathbb{C}^{\times}$-equivariant. This completes the proof of (ii).
		
		It remains only to prove (iii). The $\mathbb{C}^{\times}$-invariance of $\Gamma$ follows immediately from \eqref{Equation: Funny action} and the fact that $\tau=(f_1,\ldots,f_{\ell}):\g\longrightarrow\mathbb{C}^{\ell}$. Having previously established that $\pi:\widetilde{\g\times\mathbb{C}^{\ell}}\longrightarrow\g\times\mathbb{C}^{\ell}$ is $\mathbb{C}^{\times}$-equivariant, we see that $\mathcal{E}_{\Gamma}(\g\times\c)=\pi^{-1}(\Gamma)$ is also $\mathbb{C}^{\times}$-invariant.  
	\end{proof}
	
	\begin{remark}
		Note that the $\mathbb{C}^{\times}$-action on $\mathfrak{c}$ depends on a choice of principal $\mathfrak{sl}_2$-triple. The $\mathbb{C}^{\times}$-action on $\widetilde{\g\times\mathfrak{c}}$ depends both on this choice of triple and a choice of homogeneous, algebraically independent generators of $\mathrm{S}(\g^*)^G$.  
	\end{remark}

	\subsection{Rees algebras and Poisson structures}\label{Subsection: Quantization}
	In this section, we discuss the structure of the Rees algebra corresponding to $\widetilde{\g\times\c}$. We start by recording the following general lemma.
	
	\begin{lemma}\label{Lemma: Regular sequence}
		Suppose that a commutative $\mathbb{C}$-algebra $A$ is an integral domain. If $a_1,\ldots,a_n\in A$, then the sequence $(a_1-x_1,\ldots,a_n-x_n)$ is regular in $A[x_1,\ldots,x_n]$.
	\end{lemma}
	
	\begin{proof}
		Since $A$ is an integral domain, we see that $(-x_{\sigma(1)},\ldots,-x_{\sigma(n)})$ is a regular sequence in $A[x_1,\ldots,x_{n}]$ for all $\sigma\in S_{n}$. This allows us to apply \cite[Corollary 5.3]{hem} and conclude that $(a_1-x_1,\ldots,a_{n}-x_{n})$ is a regular sequence in $A[x_1,\ldots,x_{n}]$.
	\end{proof}
	
	We have $\mathbb{C}[\g\times\mathfrak{c}]=\mathrm{S}(\g^*)\otimes\mathrm{S}(\g^*)^G$.  	Write $\mathcal{I}_{\Gamma}\subseteq\mathrm{S}(\g^*)\otimes\mathrm{S}(\g^*)^G$ for the defining ideal of the graph $\Gamma=\g\times_{\mathfrak{c}}\mathfrak{c}$ in $\g\times\c$. Consider the Rees algebra of $\mathcal{I}_{\Gamma}$, i.e. the $\mathbb{Z}_{\geq 0}$--graded $\mathrm{S}(\g^*)\otimes \mathrm{S}(\g^*)^G$-algebra
	$$\mathrm{Rees}_{\mathcal{I}_{\Gamma}}(\mathrm{S}(\g^*)\otimes\mathrm{S}(\g^*)^G)=\bigoplus_{j=0}^{\infty}(\mathcal{I}_{\Gamma})^jt^j\subseteq(\mathrm{S}(\g^*)\otimes\mathrm{S}(\g^*)^G)[t].$$ As discussed in Section \ref{Subsection: Polishchuk}, we have that $\widetilde{\g\times\c}=\mathrm{Proj}(\mathrm{Rees}_{\mathcal{I}_{\Gamma}}(\mathrm{S}(\g^*)\otimes\mathrm{S}(\g^*)^G))$. To describe this Rees algebra more concretely, we fix a collection $f_1,\ldots,f_{\ell}$ of homogeneous, algebraically independent generators of $\mathrm{S}(\g^*)^G$ that form a regular sequence in $\mathrm{S}(\g^*)$. The ideal $I_{\Gamma}$ is generated by the elements $$f_1\otimes 1-1\otimes f_1,\ldots,f_{\ell}\otimes 1-1\otimes f_{\ell}\in\mathrm{S}(\g^*)\otimes \mathrm{S}(\g^*)^G.$$ By identifying $\mathrm{S}(\g^*)^G$ with $\mathbb{C}[x_1,\ldots,x_{\ell}]$ using $f_1,\ldots,f_{\ell}$, we can apply Lemma~\ref{Lemma: Regular sequence} to show that these generators of $I_\Gamma$ form a regular sequence in $\mathrm{S}(\g^*)\otimes\mathrm{S}(\g^*)^G$. Consider the surjective morphism 
	$$\phi:(\mathrm{S}(\g^*)\otimes\mathrm{S}(\g^*)^G)[x_1,\ldots,x_{\ell}]\longrightarrow\mathrm{Rees}_{\mathcal{I}_{\Gamma}}(\mathrm{S}(\g^*)\otimes\mathrm{S}(\g^*)^G)$$ of $\mathbb{Z}_{\geq 0}$--graded $\mathrm{S}(\g^*)\otimes\mathrm{S}(\g^*)^G$--algebras satisfying $\phi(x_i)=f_i\otimes 1-1\otimes f_i$ for all $i\in\{1,\ldots,\ell\}$. By \cite[Corollary 5.5.6]{hun-swa:06}, $$\mathrm{ker}(\phi)=\langle (f_i\otimes 1-1\otimes f_i)x_j-(f_j\otimes 1-1\otimes f_j)x_i:1\leq i<j\leq\ell\rangle.$$ It follows that $\phi$ descends to a $\mathbb{Z}_{\geq 0}$--graded $\mathrm{S}(\g^*)\otimes\mathrm{S}(\g^*)^G$-algebra isomorphism
	$$\overline{\phi}:(\mathrm{S}(\g^*)\otimes\mathrm{S}(\g^*)^G)[x_1,\ldots,x_{\ell}]\bigg/\bigg\langle (f_i\otimes 1-1\otimes f_i)x_j-(f_j\otimes 1-1\otimes f_j)x_i:1\leq i<j\leq\ell\bigg\rangle\overset{\cong}\longrightarrow\mathrm{Rees}_{\mathcal{I}_{\Gamma}}(\mathrm{S}(\g^*)\otimes\mathrm{S}(\g^*)^G).$$		
	satisfying $\overline{\phi}([x_i])=f_i\otimes 1 - 1\otimes f_i$ for all $i\in\{1,\ldots,\ell\}$. Equip the domain of $\overline{\phi}$ with the unique Poisson bracket for which $\overline{\phi}$ is a Poisson algebra isomorphism, where the Poisson structure on the codomain comes from Proposition~\ref{Proposition: Computation}(ii). This bracket is determined by the following properties:
	\begin{itemize}
		\item the Poisson bracket of two degree-zero elements is the usual Poisson bracket in $\mathrm{S}(\g^*)\otimes\mathrm{S}(\g^*)^G$ of those elements;
		\item $\{[x_i],[x_j]\}=0$ for all $i,j\in\{1,\ldots,\ell\}$.
	\end{itemize}
	
	One natural question that emerges is how to the quantize this Poisson structure. This will be the subject of forthcoming work by the authors.
	
	\bibliographystyle{acm}
	\bibliography{one-step}

\end{document}